\documentclass{article}
\usepackage{amssymb}
\usepackage{amsmath}
\usepackage{amsfonts}

\newtheorem{theorem}{Theorem}[section]

\newtheorem{corollary}[theorem]{Corollary}

\newtheorem{definition}[theorem]{Definition}

\newtheorem{lemma}[theorem]{Lemma}

\newtheorem{remark}[theorem]{Remark}

\newenvironment{proof}[1][Proof]{\noindent\textbf{#1.} }{\ \rule{0.5em}{0.5em}}
\input{tcilatex}
\begin{document}

\date{}

\title{Eigenvalue estimates via H\"{o}mander's $L^2$-method}
\author{Qingchun Ji \\ Li Lin}

\maketitle

\footnotetext{Partially
supported by NSFC 11671090.}

\begin{abstract}
Under various elliptic boundary conditions, we obtain lower  eigenvalue estimates for Dirac operators by using Hormander's weighted $L^2$-technique. Lower bounds in terms of the volume of the underlying manifold are also deduced from the sharp Sobolev inequality due to Li and Zhu(\cite{LZ}).
\end{abstract}



\section{Introduction}
In this paper, we apply H\"{o}mander's weighted $L^2$-method(\cite{H}) to study eigenvalues of Dirac operators of Dirac bundles over Riemannian manifolds. The conformal covariance(\cite{Hi}) of the classical Dirac operators played an important role in estimating eigenvalues of the classical  Dirac operator on spin manifolds, see \cite{H}-\cite{HMZ}, \cite{FK}, \cite{CWZ} and the references therein. Different from spinor bundles over spin manifolds,  the connection of the Dirac bundle is not determined by the Levi-Civita connection of the underlying manifold. In general, we don't have conformal covariance for Dirac operators of Dirac bundles. B\"{a}r(\cite{B}) generalized the Hijazi estimate(\cite{H}) to Dirac operators of Dirac bundles over closed manifolds. To avoid the use of conformal covariance, the modified connection is the key technique in B\"{a}r's proof. We will consider eigenvalue bounds for Dirac operators of Dirac bundles under both local(Definitions \ref{Md}, \ref{Ld}) and global(Definitions \ref{Ad}, \ref{MAd}) boundary conditions. Our weighted $L^2$-identity is given by Lemma \ref{weighted-L2-prop} below where the boundary terms will be dealt with using the Morrey trick for the case of local boundary condition and the boundary Dirac operator for the case of global condition. By a rescaling argument, we also obtain lower bounds in terms of the volume of the underlying manifolds where the Li-Zhu inequality (\cite{LZ}) is the fundamental tool. Recently, Chang, Chen and Wu(\cite{CCW}) also study eigenvalue estimates in CR geometry by establishing weighted Rayleigh formula, their method is closely related to this paper.
\bigskip\

\section{Weighted $L^2$-estimates for Dirac operators}

\subsection{Dirac bundles and Dirac operators\label{Dirac-bdl}}

In this section, we recall some basic facts of the Dirac operator and set up
the notations . Let $(M,g)$ be a smooth $n$-dimensional Riemannian manifold
and $C\ell \left( M,g\right) \rightarrow M$ be the corresponding \emph{%
Clifford bundle}. Let $\mathbb{S}\rightarrow M$ be a bundle of left $C\ell
\left( M,g\right) $-modules endowed with a Riemannian metric $\langle \cdot
,\cdot \rangle $ and a Riemannian connection $\nabla $ such that at any $%
x\in M$, for any unit vector $e\in T_{x}M$ and any $s_1,s_2\in \mathbb{
S}_{x}$, 
\begin{equation}
\left\langle e\cdot s_1,e\cdot s_2\right\rangle =\left\langle
s_1,s_2\right\rangle  \label{orthogonal}
\end{equation}%
where the $\cdot $ is the Clifford multiplication. Furthermore, for any
smooth vector field $X$ on $M$ and smooth section $s$ of $\mathbb{S}$, 
\begin{equation}
\nabla \left( X\cdot s\right) =\left( \nabla X\right) \cdot s+X\cdot \left(
\nabla s\right) ,  \label{Leibniz}
\end{equation}%
where on the right hand side, $\nabla $ are the covariant derivatives of the
Levi-Civita connection on $M$ and a connection on $\mathbb{S}$ for the first
and second terms respectively.

\begin{definition}
A bundle $(\mathbb{S},\langle\cdot,\cdot\rangle,\nabla)$ of $C\ell\left(
M,g\right) $-modules satisfying $\left( \ref{orthogonal}\right) ,\left( \ref%
{Leibniz}\right) $ is called a \emph{Dirac bundle} over $M$ (Definition 5.2 
\cite{LM}) and has a canonically associated \emph{Dirac operator} $D$ such
that for any section $s$ of $\mathbb{S}$, 
\begin{equation}
Ds=\sum_{i=1}^{n}e_{i}\cdot \nabla _{e_{i}}s,  \label{DiracOper}
\end{equation}%
where $\left\{ e_{i}\right\} _{i=1}^{n}$ is any orthonormal basis of $T_{x}M$
for $x$ on $M$.
\end{definition}

On the Dirac bundle $\mathbb{S}$ we can define the \emph{canonical section} $%
\mathfrak{R}$ of $\mathrm{End}\left( \mathbb{S}\right) $, such that for any
smooth section $s$ of $\mathbb{S}$, 
\begin{equation}
\mathfrak{R}s =\frac{1}{2}\sum_{i,j=1}^{n}e_{i}\cdot e_{j}\cdot R_{e_{i},e_{j}}s,  \label{R-operator}
\end{equation}%
where $R_{e_i,e_j}s=\nabla^2s(e_i,e_j)-\nabla^2s(e_j,e_i)$ is the curvature of $\mathbb{S}$. Then we have the
Bochner formula (c.f. Theorem 8.2 in Chapter II \cite{LM}) 
\begin{equation}
D^{2}=\nabla ^{\ast }\nabla +\mathfrak{R}.  \label{Gen-Bochner}
\end{equation}%
Especially, when $\mathbb{S}$ is the \emph{spinor bundle} over a spin
manifold $M\,$, Lichnerowicz's theorem says 
\begin{equation}
\mathfrak{R}=\frac{1}{4}R\cdot \mathrm{Id}_{\mathbb{S}},  \label{spin-R}
\end{equation}%
where $R$ is the scalar curvature of $\left( M,g\right) $, and $\mathrm{Id}_{%
\mathbb{S}}$ is the identity map on $\mathbb{S}$.

For any $x\in M$, $\mathfrak{R}\left( x\right)
\in \mathrm{End}\left( \mathbb{S}_{x}\right) $ is self-adjoint , we denote 
\begin{equation}
\kappa\left( x\right) =\text{the smallest eigenvalue of }%
\mathfrak{R}\left( x\right). \label{first-eigen-fcn}
\end{equation}%
Obviously, $\kappa\left( x\right) $ is a 
\emph{Lipschitz} function on $M$. 

\bigskip

The following Lemma will be used to describe when the eigenvalue lower bounds can be achieved. Note that we don't have the Ricci identity(for the Ricci identity on spin manifolds, we refer to \cite{F} ) for general Dirac bundle.
\begin{lemma}Let $a,b,c\in\mathbb{C}$ be constants, $s\in\Gamma(M,\mathbb{S})$. Assume $\nabla_X s = aX\cdot s+bX\cdot\nabla\phi\cdot s + cX(\phi)s$ for all tangent vector $X$ and some $\phi\in C^1(M)$, then\begin{eqnarray}
\mathfrak{R}s&=&(n-1)\Bigg(\Big(na^2+(n-2)b^2\vert\nabla\phi\vert^2+b\Delta\phi\Big)s+2ab\nabla\phi\cdot s\Bigg).\label{RI}
\end{eqnarray}
\end{lemma}
\begin{proof}
We work with an orthonormal frame $\{e_i\}_{i=1}^n$ which is normal at a given point, and denote $\phi_i = e_i(\phi), \phi_{ij}=\langle\nabla_i\phi,e_j\rangle$ for $1\leq i,j\leq n$. We will compute $\nabla^2s(e_i,e_j)$ by the assumption and moving $e_i,e_j$ to the left. At the given point, we have 
\begin{eqnarray*}
\nabla^2s(e_i,e_j)&=& \nabla_{e_i}\nabla_{e_j}s \\ &=& \left(ae_j\cdot+be_j\cdot\nabla\phi\cdot+c\phi_j\right)\left(ae_i\cdot s+be_i\cdot\nabla\phi\cdot s+c\phi_is\right) \\ &&+be_j\cdot\nabla_{e_i}\nabla\phi\cdot s+c\phi_{ij}s\\&=& \left(a^2+b^2\vert\nabla\phi\vert^2\right)e_j\cdot e_i\cdot s +be_j\cdot\nabla_{e_i}\nabla\phi\cdot s + \left(c\phi_{ij}+c^2\phi_i\phi_j\right)s\\ && +\big(a(c-2b)\phi_ie_j+ac\phi_je_i+bc\phi_j e_i\cdot\nabla\phi+b(c-2b)\phi_ie_j\cdot\nabla\phi\big)\cdot s 
\end{eqnarray*}
which gives
\begin{eqnarray*}
\mathfrak{R}s&=&\frac{1}{2}\sum_{i,j=1}^{n}e_{i}\cdot e_{j}\cdot (\nabla^2s(e_i,e_j)-\nabla^2s(e_j,e_i)) \\ &=& \sum_{i\neq j}e_{i}\cdot e_{j}\cdot \nabla^2s(e_i,e_j)\\ &=&(n-1)\Big(\big(na^2+(n-2)b^2\vert\nabla\phi\vert^2+b\Delta\phi\big)s+2ab\nabla\phi\cdot s\Big),
\end{eqnarray*}where we have used in the last equality \begin{eqnarray*}
\sum_{i\neq j}e_{i}\cdot e_{j}\cdot e_j\cdot\nabla_{e_i}\nabla\phi\cdot s&=&-(n-1)\sum_{i=1}^ne_i\cdot\nabla_{e_i}\nabla\phi\cdot s\\ &=&-(n-1)\sum_{i,j=1}^n\phi_{ij}e_i\cdot e_j\cdot s\\ &=& (n-1)\Delta\phi s.
\end{eqnarray*}
The proof is complete.
\end{proof}

\subsection{Weighted $L^{2}$-estimates\label{proofs}}

Let $\mathbb{S}$ be a Dirac bundle over a Riemannian manifold $\left(
M,g\right)$(with or without boundary) and $D$ be the Dirac operator.

Given $s_1,s_2\in\Gamma (M,\mathbb{S})$, since  
\begin{eqnarray}
{\rm div}\left(\sum_{i=1}^n\langle e_i\cdot s_1, s_2\rangle e_i\right)&=&\langle Ds_1,s_2\rangle -\langle s_1,Ds_2\rangle \label{I1}, 
\end{eqnarray}
we have 
\begin{eqnarray}
\int_M\langle Ds_1,s_2\rangle = \int_M\langle s_1,Ds_2\rangle + \int_{\partial{M}}\langle \nu\cdot s_1,s_2\rangle\label{I1'}
\end{eqnarray}where $\nu$ is the outward unit normal vector field of $\partial{M}$, we require that one of $s_1$ and $s_2$ has compact support if the underlying manifold $M$ is non-compact.

Similarly, for $\theta\in\Gamma(M,T^*M\otimes\mathbb{S})$ and $s\in\Gamma(M,\mathbb{S})$
\begin{eqnarray}
{\rm div}\left(\sum_{i=1}^n\langle \theta(e_i), s\rangle e_i\right)&=&\langle \theta,\nabla s\rangle -\langle \nabla^*\theta,s\rangle \label{I2}
\end{eqnarray}
where $\nabla^*\theta:=-\sum_{i=1}^n(\nabla_{e_i}\theta)(e_i).$ If we assume as above that one of $\theta$ and $s$ is compactly supported when $M$ is non-compact, then
\begin{eqnarray}
\int_M\langle \nabla^*\theta,s\rangle = \int_M\langle\theta,\nabla s\rangle - \int_{\partial{M}}\langle \theta(\nu),s\rangle\label{I2'}
\end{eqnarray}

Integrate $(\ref{Gen-Bochner})$ by using (\ref{I1'}) and (\ref{I2'}), 
\begin{eqnarray}
\int_{M}\left\vert Ds\right\vert ^{2}&=&\int_{M}\left( \left\vert \nabla
s\right\vert ^{2}+\left\langle s,\mathfrak{R}s\right\rangle \right)-\int_{\partial{M}}\left\langle \nu\cdot Ds+\nabla_\nu s,s\right\rangle\notag \\ &=& \int_{M}\left( \left\vert \nabla
s\right\vert ^{2}+\left\langle s,\mathfrak{R}s\right\rangle \right)-\sum_{\alpha=1}^{n-1}\int_{\partial{M}}\left\langle \nu\cdot e_\alpha\cdot\nabla_{e_\alpha}s,s\right\rangle .
\label{L2-identity}
\end{eqnarray}
where we have used the adapted frame along $\partial{M}$, i.e., $e_n = \nu$.

\bigskip We introduce a twistor operator for every $\eta\in\mathbb{R}^2$ \begin{equation}P_\eta(X,s):=\eta_1\nabla_Xs + \eta_2 X\cdot Ds.\label{Pen}\end{equation}
where $X\in\Gamma(M,TM), s\in\Gamma(M,\mathbb{S})$. For a given $s\in\Gamma(M,\mathbb{S})$, we denote $P_\eta s:=P_\eta(\cdot,s)\in\Gamma(M,T^*M\otimes\mathbb{S}).$ By definition, we have\begin{equation}\sum_{i=1}^ne_i\cdot P_\eta(e_i,s)=(\eta_1-n\eta_2)Ds\label{P}
\end{equation} and
\begin{equation}\vert P_\eta s\vert^2=\eta_1^2\vert\nabla s\vert^2 + \eta_2(n\eta_2-2\eta_1)\vert Ds\vert^2.\label{P0}
\end{equation}for any $s\in\Gamma(M,\mathbb{S})$. According to (\ref{P0}), we can rewrite the identity (\ref{L2-identity}) as follows which will be the starting point of our weighted estimate.
\begin{equation}\|\eta\|^2\int_{M}\left\vert Ds\right\vert ^{2} =  \int_{M}\left( \left\vert P_\eta s\right\vert ^{2}+\eta_1^2\left\langle s,\mathfrak{R}s\right\rangle \right)-\eta_1^2\sum_{\alpha=1}^{n-1}\int_{\partial{M}}\left\langle \nu\cdot e_\alpha\cdot\nabla_{e_\alpha}s,s\right\rangle .
\label{unw}
\end{equation}for all $s\in\Gamma(M,\mathbb{S})$, where $\eta\in\mathbb{R}^2$ and $\|\eta\|^2:=\eta_1^2-2\eta_1\eta_2+n\eta_2^2$.

\bigskip

In what follows, we will work with weighted $L^2$-spaces.
\begin{definition}
(Weighted $L^{2}$-space)\label{weighted-L2-space} Let $\varphi :M\rightarrow 
\mathbb{R}$ be a $C^{2}$ function. For any sections $s_1$ and $s_2$ of 
$\mathbb{S}$, let the \emph{weighted }inner product of $s$ and $s^{\prime }$
be%
\begin{equation*}
\left( s_1,s_2\right) _{\varphi }=\int_{M}\left\langle s_1,s_2\right\rangle e^{-\varphi }.
\end{equation*}
where $\varphi :M\rightarrow \mathbb{R}$ is a $C^2$ function. Let $%
\left\Vert s\right\Vert _{\varphi }=\sqrt{\left(s,s\right) _{\varphi }} $
and denote by $L_{\varphi }^{2}\left( M,\mathbb{S}\right) $ be the space of
sections $s$ of $\mathbb{S}$ such that $\left\Vert s\right\Vert _{\varphi
}<\infty $. We will drop the subscript $\varphi $ when $\varphi =0$.
\end{definition}

For the Dirac operator $D:L_{\varphi }^{2}\left( M,\mathbb{S}\right)
\rightarrow L_{\varphi }^{2}\left( M,\mathbb{S}\right) $, let $D_{\varphi
}^{\ast }$ be its formal adjoint with respect to the measure $e^{-\varphi
}dvol_{g}$. For $D_{\varphi }^{\ast }$, we have the following identity which
is immediate from definition.

\begin{equation}
D_{\varphi }^{\ast }s=e^{\varphi }D\left( e^{-\varphi }s\right) =-\nabla
\varphi \cdot s+Ds.  \label{weighted-Dirac-relation}
\end{equation}

\bigskip

We will make use of the following weighted twistor operator $P_{\eta,\phi}$ for any $\phi\in C^1(M)$ and $\eta\in\mathbb{R}^2$ $$P_{\eta,\phi}(X,s):=P_\eta(e^\phi X,e^{-\phi}s).$$ From $e^\phi \circ \nabla\circ e^{-\phi} = \nabla - X(\phi)$ and (\ref{weighted-Dirac-relation}), it follows immediately 
\begin{equation}
P_{\eta,\phi} (X,s) = P_\eta (X,s) -\eta_1X(\phi)s -\eta_2 X\cdot\nabla\phi\cdot s\label{P1}
\end{equation}for any $X\in\Gamma(M,TM)$ and $s\in\Gamma(M,\mathbb{S})$.
We also adopt the notation $P_{\eta,\phi} s:=P_{\eta,\phi}(\cdot,s)\in\Gamma(M,T^*M\otimes\mathbb{S})$ for a given $s\in\Gamma(M,\mathbb{S})$.

\bigskip

Let $\Delta $ be the Laplace-Beltrami operator for functions on $\left( M,g\right) $. Here is an identity between twistor operators with different weights.

\begin{lemma}
Let $\phi\in C^1(M)$ be real function and $\eta\in \mathbb{R}^2$, then we have
\begin{eqnarray}
\vert P_{\eta}s\vert^2+2(\eta_1^2-\|\eta\|^2){\rm Re}\left\langle\nabla\phi\cdot s,Ds\right\rangle &=&  \vert P_{\eta,\phi}s\vert^2 +\eta_1^2 {\rm div}\left(\vert s\vert^2\nabla\phi\right) \label{PP2} \\ && - \big(\eta_1^2\Delta\phi +  \|\eta\|^2\vert\nabla\phi\vert^2 \big)\vert s\vert^2
\notag\end{eqnarray}
holds for all $s\in\Gamma(M,\mathbb{S})$ where $\|\eta\|^2:=\eta_1^2-2\eta_1\eta_2+n\eta_2^2.$ 
\end{lemma}
\begin{proof}
We prove (\ref{PP2}) by a derect computation of the norm of $P_{\eta,\phi}s$.
\begin{eqnarray}\vert P_{\eta,\phi}s\vert^2 &=& \sum_{i=1}^n\vert P_\eta(e_i,s)-\eta_1e_i(\phi)s-\eta_2 e_i\cdot\nabla\phi\cdot s\vert^2 \ \ {\rm by} \ \ (\ref{P1})\notag \\ &=& \vert P_\eta s\vert^2  +\eta_1^2\vert\nabla\phi\vert^2\vert s\vert^2+ n\eta_2^2\vert\nabla\phi\vert^2\vert s\vert^2+2\eta_2(\eta_1-n\eta_2){\rm Re}\langle Ds,\nabla\phi\cdot s\rangle \  \ {\rm by} \ (\ref{P})\notag \\ && +2{\rm Re}\sum_{i=1}^n\big(\eta_1\eta_2\langle e_i\cdot\nabla\phi\cdot s,e_i(\phi)s\rangle-\eta_1\langle P_\eta(e_i,s),e_i(\phi)s\rangle\big)
\notag \\ &=& \vert P_\eta s\vert^2 +\|\eta\|^2\vert\nabla\phi\vert^2\vert s\vert^2 +2\eta_2(\eta_1-n\eta_2){\rm Re}\langle Ds,\nabla\phi\cdot s\rangle\notag \\&& -2\eta_1{\rm Re}\langle \eta_1\nabla_{\nabla\phi}s+\eta_2\nabla\phi\cdot Ds,s\rangle\notag \\ &=& \vert P_\eta s\vert^2 +\|\eta\|^2\vert\nabla\phi\vert^2\vert s\vert^2 +2\eta_2(2\eta_1-n\eta_2){\rm Re}\langle Ds,\nabla\phi\cdot s\rangle\notag \\&&  + \eta_1^2\vert s\vert^2 \Delta\phi-\eta_1^2{\rm div}(\vert s\vert^2\nabla\phi).\notag 
\end{eqnarray}
This is the desired identity (\ref{PP2}). 
\end{proof}

\bigskip

By the argument in \cite{JZ1}, we can introduce weights into (\ref{unw}) as follows.

\begin{lemma}
\label{weighted-L2-prop}Let $\mathbb{S}$ be a Dirac boundle over an $n$-dimensional Riemannian manifold $(M,g)$ which is compact with smooth boundary, then for all $s\in\Gamma(M,\mathbb{S})$, real functions $\varphi\in C^2(M)$ and $\tau,\delta\in\mathbb{R},\eta\in\mathbb{R}^2\setminus\{0\}$,
we have
\begin{eqnarray}
\left\Vert D_{\tau\varphi }^{\ast }s\right\Vert_{2(\tau-\delta)\varphi}^{2}&=& \frac{(\xi-1)\delta}{\xi}\int_M\left(\Delta\varphi+\frac{(\xi+1)\delta}{\xi}\vert\nabla\varphi\vert^2\right)\vert s\vert^2e^{2(\delta-\tau)\varphi} \notag\\ && +(1-\xi)\int_M\left(\frac{1}{\eta_1^2}\vert P_{\eta,(\tau+(\frac{1}{\xi}-1)\delta)\varphi}s\vert^2+\langle s,\mathfrak{R}s\rangle\right)e^{2(\delta-\tau)\varphi}
\notag \\&&+(\xi-1)\sum_{\alpha=1}^{n-1}\int_{\partial{M}}\left\langle \nu\cdot e_\alpha\cdot\nabla_{e_\alpha}s,s\right\rangle e^{2(\delta-\tau)\varphi}\notag \\ && 
+(\tau-\delta)(1-\xi)\int_{\partial{M}}\left\langle \nu\cdot \nabla\varphi \cdot s,s\right\rangle e^{2(\delta-\tau)\varphi}\notag \\ &&+\Big((1-\xi)\tau+(\frac{1}{\xi}+\xi-2)\delta\Big)\int_{\partial{M}}\nu(\varphi)\vert s\vert^2 e^{2(\delta-\tau)\varphi} \label{weighted-L2-est1}
\end{eqnarray}
where $\xi:=\frac{n\eta_2^2-2\eta_1\eta_2}{\eta_1^2-2\eta_1\eta_2+n\eta_2^2}$, $\{e_\alpha\}_{\alpha =1}^{n-1}$ is an orthogonal frame of $\partial{M}$ and $\nu$ is the outward unit normal vector field of $\partial{M}$. 
\end{lemma}

\begin{proof}
Set $t :=e^{(\delta-\tau)\varphi }s$, then we know by (\ref{unw}) and (\ref{weighted-Dirac-relation}) 
\begin{eqnarray}
\|\eta\|^2\left\Vert D_{\tau\varphi }^{\ast }s\right\Vert_{2(\tau-\delta)\varphi}^{2}
&=&\|\eta\|^2\int_{M}\left\vert e^{\tau\varphi }D\left( e^{-\tau\varphi }s\right) \right\vert
^{2}e^{2(\delta-\tau)\varphi } \notag   \\ &=&\|\eta\|^2 \int_{M}\left\vert D_{\delta\varphi}^*t\right\vert
^{2} \notag   \\
&=&\|\eta\|^2\int_{M}\left\vert Dt -\delta\nabla \varphi \cdot t
\right\vert ^{2}  \notag \\
&=&\|\eta\|^2\int_{M}\left(\delta^2\left\vert \nabla
\varphi \right\vert ^{2}\left\vert t \right\vert ^{2}-2\delta\mathrm{Re}
\left\langle \nabla \varphi \cdot t ,Dt \right\rangle \right) , 
\notag \\ && +\int_{M}\left( \left\vert P_\eta t\right\vert ^{2}+\eta_1^2\left\langle t,\mathfrak{R}t\right\rangle \right)-\eta_1^2\sum_{\alpha=1}^{n-1}\int_{\partial{M}}\left\langle \nu\cdot e_\alpha\cdot\nabla_{e_\alpha}t,t\right\rangle \notag \\ &=&\int_{M}\left(\|\eta\|^2\delta^2\left\vert \nabla
\varphi \right\vert ^{2}\left\vert t \right\vert ^{2}+\eta_1^2\left\langle t,\mathfrak{R}t\right\rangle \right) 
+\int_{M}\left( \left\vert P_\eta t\right\vert ^{2}-2\|\eta\|^2\delta\mathrm{Re}
\left\langle \nabla \varphi \cdot t ,Dt \right\rangle\right)\notag \\ && -\eta_1^2\sum_{\alpha=1}^{n-1}\int_{\partial{M}}\left\langle \nu\cdot e_\alpha\cdot\nabla_{e_\alpha}t,t\right\rangle.\label{1}
\end{eqnarray}
Choosing $\phi=\frac{\delta\varphi}{1-\big(\frac{\eta_1}{\|\eta\|}\big)^2}=\frac{\delta\varphi}{\xi}$ in (\ref{PP2}), it follows that
\begin{eqnarray}
\int_{M}\left(\left\vert P_\eta t\right\vert ^{2}-2\|\eta\|^2\delta\mathrm{Re}
\left\langle \nabla \varphi \cdot t ,Dt \right\rangle \right) &=&-\int_M \left(\frac{\eta_1^2\delta}{\xi}\Delta\varphi + \frac{ \|\eta\|^2\delta^2}{\xi^2} \vert\nabla\varphi\vert^2 \right)\vert t\vert^2 \notag \\ && +\int_M \vert P_{\eta,\frac{\delta}{\xi}\varphi}t\vert^2 + \frac{\eta_1^2\delta}{\xi}\int_{\partial{M}}\nu(\varphi)\vert t\vert^2 \notag \\ &=&-\int_M \left(\frac{\eta_1^2\delta}{\xi}\Delta\varphi + \frac{ \|\eta\|^2\delta^2}{\xi^2} \vert\nabla\varphi\vert^2 \right)\vert s\vert^2e^{2(\delta-\tau)\varphi} \notag \\ && +\int_M \vert P_{\eta,(\tau+(\frac{1}{\xi}-1)\delta)\varphi}s\vert^2e^{2(\delta-\tau)\varphi} +  \frac{\eta_1^2\delta}{\xi}\int_{\partial{M}}\nu(\varphi)\vert s\vert^2e^{2(\delta-\tau)\varphi}. \notag
\end{eqnarray}
For the boundary term, we have 
\begin{eqnarray}
-\sum_{\alpha=1}^{n-1}\int_{\partial{M}}\left\langle \nu\cdot e_\alpha\cdot\nabla_{e_\alpha}t,t\right\rangle &=& -\sum_{\alpha=1}^{n-1}\int_{\partial{M}}\left\langle \nu\cdot e_\alpha\cdot\nabla_{e_\alpha}s,s\right\rangle e^{2(\delta-\tau)\varphi} \notag \\ && 
+(\tau-\delta)\int_{\partial{M}}\left\langle \nu\cdot \left(\nabla\varphi - \nu(\varphi)\nu\right)\cdot s,s\right\rangle e^{2(\delta-\tau)\varphi}\notag \\ &=& -\sum_{\alpha=1}^{n-1}\int_{\partial{M}}\left\langle \nu\cdot e_\alpha\cdot\nabla_{e_\alpha}s,s\right\rangle e^{2(\delta-\tau)\varphi} \notag \\ && 
+(\tau-\delta)\int_{\partial{M}}\left(\left\langle \nu\cdot \nabla\varphi \cdot s,s\right\rangle +\nu(\varphi)\vert s\vert^2\right) .\notag
\end{eqnarray}
Plugging the above identities into (\ref{1}), we obtain
the desired inequality. \end{proof}

As a function on $\mathbb{P}^1$, $\xi(\eta):=\frac{n\eta_2^2-2\eta_1\eta_2}{\eta_1^2-2\eta_1\eta_2+n\eta_2^2}$ has a maximun of 1 and a minimum of $\frac{-1}{n-1}$. The minimum value $\frac{-1}{n-1}$ is reached exactly along the line $\eta_1=n\eta_2$. When $\eta_1=n\eta_2$, $P_\eta = \eta_1 P$ where $P$ is the standard twistor operator $P(X,s):=P_{(1,\frac{1}{n})}(X,s)=\nabla_Xs + \frac{1}{n}X\cdot Ds$. As a direct consequence of Lemma \ref{weighted-L2-prop}, we have
\begin{corollary}
\label{2-dim}Let $\mathbb{S}$ be a Dirac boundle over an $n$-dimensional Riemannian manifold $(M,g)$ which is compact with smooth boundary, then for all $s\in\Gamma(M,\mathbb{S})$, real functions $\varphi\in C^2(M)$ and constants $\tau,\delta\in\mathbb{R}$,
we have
\begin{eqnarray}
\left\Vert D_{\tau\varphi }^{\ast }s\right\Vert_{2(\tau-\delta)\varphi}^{2}&=& n\delta\int_M\big(\Delta\varphi+(2-n)\delta\nabla\varphi\vert^2\big)e^{2(\delta-\tau)\varphi}
\notag \\&& +\frac{n}{n-1}\int_M\big(\vert P_{(\tau-n\delta)\varphi}s\vert^2+\langle s,\mathfrak{R}s\rangle\big)e^{2(\delta-\tau)\varphi}
\notag \\&&-\frac{n}{n-1}\sum_{\alpha=1}^{n-1}\int_{\partial{M}}\left\langle \nu\cdot e_\alpha\cdot\nabla_{e_\alpha}s,s\right\rangle e^{2(\delta-\tau)\varphi}\notag \\ && 
+\frac{n(\tau-\delta)}{n-1}\int_{\partial{M}}\left\langle \nu\cdot \nabla\varphi \cdot s,s\right\rangle e^{2(\delta-\tau)\varphi}\notag \\ &&+\frac{n(\tau-n\delta)}{n-1}\int_{\partial{M}}\nu(\varphi)\vert s\vert^2 e^{2(\delta-\tau)\varphi}. \label{weighted-L2-est11}
\end{eqnarray}
\end{corollary}

The next identity allows $\delta=0$ in (\ref{weighted-L2-est1}) with the second order term $\Delta\varphi$ being preserved.
\begin{corollary}
\label{weighted-L2-coro}Let $\mathbb{S}$ be a Dirac boundle over an $n$-dimensional Riemannian manifold $(M,g)$ which is compact with smooth boundary, then for all $s\in\Gamma(M,\mathbb{S})$, real functions $\varphi\in C^2(M)$ and constants $\tau,r\in\mathbb{R}(r\neq 0)$,
we have
\begin{eqnarray}
\left\Vert D_{\tau\varphi }^{\ast }s\right\Vert_{2\tau\varphi}^{2}&=& -r\int_M\left(\Delta\varphi+r\vert\nabla\varphi\vert^2\right)\vert s\vert^2e^{-2\tau\varphi} \notag\\ && +\int_M\left(\vert \nabla s-(\tau+r)d\varphi\otimes s\vert^2+\langle s,\mathfrak{R}s\rangle\right)e^{-2\tau\varphi}\notag \\&&-\sum_{\alpha=1}^{n-1}\int_{\partial{M}}\left\langle \nu\cdot e_\alpha\cdot\nabla_{e_\alpha}s,s\right\rangle e^{-2\tau\varphi}+\tau\int_{\partial{M}}\left\langle \nu\cdot \nabla\varphi \cdot s,s\right\rangle e^{-2\tau\varphi}\notag \\ &&+(\tau+r)\int_{\partial{M}}\nu(\varphi)\vert s\vert^2 e^{-2\tau\varphi}.\label{weighted-L2-est2}
\end{eqnarray}
\end{corollary}
\begin{proof}
Fix some sequence $0\neq\delta_\ell\rightarrow 0$ as $\ell\rightarrow \infty$. For sufficiently large $\ell$, one can find $\eta_{2,\ell} \rightarrow 0$ such that $\xi_\ell:=\xi(1,\eta_{2,\ell})=\frac{\delta_\ell}{r}$ for each $\ell$. Lemma \ref{weighted-L2-prop} applied to $\delta=\delta_\ell, \eta=(1,\eta_{2,\ell})$ gives
\begin{eqnarray}
\left\Vert D_{\tau\varphi }^{\ast }s\right\Vert_{2(\tau-\delta_\ell)\varphi}^{2}&=& (\delta_\ell-r)\int_M\left(\Delta\varphi+(\delta_\ell+r)\vert\nabla\varphi\vert^2\right)\vert s\vert^2e^{2(\delta_\ell-\tau)\varphi} \notag\\ && +(1-\frac{\delta_\ell}{r})\int_M\left(\vert P_{(1,\eta_{2,\ell}),(\tau+r-\delta_\ell)\varphi}s\vert^2+\langle s,\mathfrak{R}s\rangle\right)e^{2(\delta_\ell-\tau)\varphi}\notag \\&&+(\frac{\delta_\ell}{r}-1)\sum_{\alpha=1}^{n-1}\int_{\partial{M}}\left\langle \nu\cdot e_\alpha\cdot\nabla_{e_\alpha}s,s\right\rangle e^{2(\delta_\ell-\tau)\varphi}\notag \\ && 
+(\tau-\delta_\ell)(1-\frac{\delta_\ell}{r})\int_{\partial{M}}\left\langle \nu\cdot \nabla\varphi \cdot s,s\right\rangle e^{2(\delta_\ell-\tau)\varphi}\notag \\ &&+\big((1-\frac{\delta_\ell}{r})\tau+r+\frac{\delta_\ell^2}{r}-2\delta_\ell\big)\int_{\partial{M}}\nu(\varphi)\vert s\vert^2 e^{2(\delta_\ell-\tau)\varphi}.\notag
\end{eqnarray}By taking limit as $\ell\rightarrow \infty$, we have 
\begin{eqnarray}
\left\Vert D_{\tau\varphi }^{\ast }s\right\Vert_{2\tau\varphi}^{2}&=& -r\int_M\left(\Delta\varphi+r\vert\nabla\varphi\vert^2\right)\vert s\vert^2e^{-2\tau\varphi} \notag\\ && +\int_M\left(\vert P_{(1,0),(\tau+r)\varphi}s\vert^2+\langle s,\mathfrak{R}s\rangle\right)e^{-2\tau\varphi}\notag \\&&-\sum_{\alpha=1}^{n-1}\int_{\partial{M}}\left\langle \nu\cdot e_\alpha\cdot\nabla_{e_\alpha}s,s\right\rangle e^{-2\tau\varphi}+\tau\int_{\partial{M}}\left\langle \nu\cdot \nabla\varphi \cdot s,s\right\rangle e^{-2\tau\varphi}\notag \\ &&+(\tau+r)\int_{\partial{M}}\nu(\varphi)\vert s\vert^2 e^{-2\tau\varphi}.\notag
\end{eqnarray}
Now the identity (\ref{weighted-L2-est2}) follows from the fact $P_{(1,0),(\tau+r)\varphi}s=\nabla s-(\tau+r)d\varphi\otimes s$(by (\ref{Pen}) and (\ref{P1})).
\end{proof}

\section{Boundary conditions}
In this section, we restrict to compact manifolds with smooth boundary(possibly empty). When the boundary is non-empty, we have to introduce elliptic boundary conditions to make the spectrum of the Dirac operator discrete with finite dimensional eigenspaces(\cite{HMR}). We will recall some well-known boundary conditions which are originally introduced for study of the classical Dirac operator on spin manifolds, the ellipticity can be verified in the same way.

The following Riemannian version of MIT bag boundary condition was first introduced in \cite{HMR}.
\begin{definition}\label{Md}
The MIT bag boundary condition for a section $s\in \Gamma(\partial{M},\mathbb{S})$ means  $$\nu\cdot s = \sqrt{-1}s \ {\rm or} \  -\sqrt{-1}s$$ where $\nu$ is the outward unit vector normal to $\partial{M}$.
\end{definition}

By (\ref{I1'}), it is easy to see that every eigenvalue $\lambda(D)$ under the MIT bag boundary condition has a nonzero imaginary part.

We will treat the boundary term in (\ref{weighted-L2-est1}) by an analogue of the Morrey trick for the $\bar{\partial}$-equation. Let $\{e_i\}_{i=1}^n$ be an adapted orthogonal frame of the boundary, namely $e_n=\nu$ the outward unit normal vector field of the boundary.

For a section $s\in \Gamma(\partial{M},\mathbb{S})$ satisfying the MIT bag boundary condition, say $\nu\cdot s = \sqrt{-1}s$ on $\partial{M}$. By taking tangential derivatives, we know that  $$\nabla_{e_\alpha}\nu\cdot s+\nu\cdot\nabla_{e_\alpha} s=\sqrt{-1}\nabla_{e_\alpha}s$$ holds on $\partial{M}$ for each $\alpha =1,\cdots,n-1$. As a consequence, we get
\begin{eqnarray*}
\sqrt{-1}\sum_{\alpha=1}^{n-1}\langle\nu\cdot e_{\alpha}\cdot\nabla_{e_\alpha} s,s\rangle &=& \sum_{\alpha=1}^{n-1}\langle\nu\cdot e_{\alpha}\cdot (\nabla_{e_\alpha}\nu\cdot s+\nu\cdot\nabla_{e_\alpha} s),-\sqrt{-1}\nu\cdot s\rangle \\ &=& \sqrt{-1}\sum_{\alpha,\beta=1}^{n-1}\langle\nabla_{e_\alpha}\nu,e_{\beta}\rangle\langle e_{\alpha}\cdot e_{\beta}\cdot s,s\rangle -\sqrt{-1}\sum_{\alpha=1}^{n-1}\langle\nu\cdot e_{\alpha}\cdot\nabla_{e_\alpha} s,s\rangle \\ &=& -\sqrt{-1} (n-1)H|s|^2 -\sqrt{-1}\sum_{\alpha=1}^{n-1}\langle\nu\cdot e_{\alpha}\cdot\nabla_{e_\alpha} s,s\rangle,
\end{eqnarray*} which implies
\begin{equation}
\sum_{\alpha=1}^{n-1}\langle\nu\cdot e_{\alpha}\cdot\nabla_{e_\alpha} s,s\rangle=-\frac{(n-1)H}{2}|s|^2, \label{b1}
\end{equation}where $H$ is the mean curvature of $\partial{M}$ w.r.t. the outward unit vector $\nu$, i.e. $H=\frac{1}{n-1}\sum_{\alpha=1}^{n-1}\langle\nabla_{e_\alpha}\nu,e_{\alpha}\rangle$.
The definition of the MIT bag boundary condition also gives
\begin{eqnarray*}
\nu\cdot s = \pm\sqrt{-1}s &\Rightarrow& \nu\cdot(e_{\alpha}\cdot s) = \mp\sqrt{-1}e_{\alpha}\cdot s \\ &\Rightarrow&\langle e_{\alpha}\cdot s,s\rangle=0, 1\leq \alpha\leq n-1,
\end{eqnarray*} which implies 
\begin{equation}
\langle\nu\cdot\nabla\varphi\cdot s, s\rangle = -\nu(\varphi)|s|^2 \label{b2}
\end{equation}

\bigskip

To describe the local boundary condition introduced in \cite{HMZ}, we the need to consider the Dirac bundle structure of the restriction $\bar{\mathbb{S}}$ of $\mathbb{S}$ on $\partial{M}$. $\bar{\mathbb{S}}$ has a natural structure of Dirac bundle over $\partial{M}$(equipped with the induced metric):
\begin{itemize}
\item The metric on $\bar{\mathbb{S}}$ is the restriction of the metric on $\mathbb{S}$.
\item The Clifford multiplication on $\bar{\mathbb{S}}$: $X\ \bar{\cdot} \:= \nu\cdot X\cdot$ for any tangent vector $X$ of $\partial{M}$.
\item The connection on $\bar{\mathbb{S}}$: $\bar{\nabla}_X:=\nabla_X - \frac{1}{2}\nabla_X\nu\ \bar{\cdot}$ where  $X$ is  tangent to $\partial{M}$.
\end{itemize}

\begin{definition}\label{Ld}
Assume that $\bar{\mathbb{S}}=\bar{\mathbb{S}}^+\oplus\bar{\mathbb{S}}^-$ is a  parallel orthogonal decomposition satisfying 
$$T\partial{M} \ \bar{\cdot} \ \bar{\mathbb{S}}^{\pm}\subseteq\bar{\mathbb{S}}
^{\mp},$$ $$\nu \cdot \bar{\mathbb{S}}^{\pm}\subseteq\bar{\mathbb{S}}
^{\mp}.$$The local boundary condition for a section $s$ of $\bar{\mathbb{S}}$ means that either $s^+=0$ or $s^-=0$ holds on $\partial{M}$.
\end{definition}

\begin{remark} The local boundary condition can be equivalently described in terms of the boundary chirality operator as in \cite{HMZ}. 
\end{remark}
The identity (\ref{I1'}) shows that every eigenvalue $\lambda(D)$ under the local boundary condition is real.
Obviously, (\ref{b1}) and (\ref{b2}) still hold for every section satisfying the local boundary condition. We have actually proved the following variants of (\ref{weighted-L2-est1}) and (\ref{weighted-L2-est11}).
\begin{lemma}\label{ML}Suppose the restriction of $s\in\Gamma(M,\mathbb{S})$ to $\partial{M}$ satisfies the MIT bag boundary condition or the local boundary condition, then we have
\begin{eqnarray}
\left\Vert D_{\tau\varphi }^{\ast }s\right\Vert_{2(\tau-\delta)\varphi}^{2}&=& \frac{(\xi-1)\delta}{\xi}\int_M\left(\Delta\varphi+\frac{(\xi+1)\delta}{\xi}\vert\nabla\varphi\vert^2\right)\vert s\vert^2e^{2(\delta-\tau)\varphi} \notag\\ && +(1-\xi)\int_M\left(\frac{1}{\eta_1^2}\vert P_{\eta,(\tau+(\frac{1}{\xi}-1)\delta)\varphi}s\vert^2+\langle s,\mathfrak{R}s\rangle\right)e^{2(\delta-\tau)\varphi}
\notag \\&&+(1-\xi)\int_{\partial{M}}\Big(\frac{(n-1)H}{2}+\frac{\delta\nu(\varphi)}{\xi}\Big)\vert s\vert^2e^{2(\delta-\tau)\varphi}.\label{W1}
\end{eqnarray}
If we choose $\eta_1=n\eta_2\neq 0$, then \begin{eqnarray}
\left\Vert D_{\tau\varphi }^{\ast }s\right\Vert_{2(\tau-\delta)\varphi}^{2}&=& n\delta\int_M\big(\Delta\varphi+(2-n)\delta\nabla\varphi\vert^2\big)e^{2(\delta-\tau)\varphi}
\notag \\&& +\frac{n}{n-1}\int_M\big(\vert P_{(\tau-n\delta)\varphi}s\vert^2+\langle s,\mathfrak{R}s\rangle\big)e^{2(\delta-\tau)\varphi}
\notag \\ && +n\int_{\partial{M}}\left(\frac{H}{2}-\delta\nu(\varphi)\right)\vert s\vert^2e^{2(\delta-\tau)\varphi}.\label{W1'}
\end{eqnarray}
\end{lemma}

\bigskip

The APS(Atiyah-Patodi-Singer) boundary condition plays an important role in the index theory for the Dirac operator, the modified APS condition was introduced in \cite{HMR}. Chen generalized the APS boundary condition in \cite{C}.
\begin{definition}\label{Ad}
The $b$-APS(Atiyah-Patodi-Singer) boundary condition for $s\in L^2(\partial{M},\bar{\mathbb{S}})$ means that $s$ belongs to the space spanned by eigensections with eigenvalues $\leq b$ of $\bar{D}=\sum_{\alpha=1}^{n-1}e_{\alpha}\ \bar{\cdot} \ \bar{\nabla}_{\alpha}:L^2(\partial{M},\bar{\mathbb{S}})\rightarrow L^2(\partial{M},\bar{\mathbb{S}})$ where $b\in\mathbb{R}$.
\end{definition}

Since \begin{eqnarray}
\sum_{\alpha=1}^{n-1}\nu\cdot e_{\alpha}\cdot\nabla_{\alpha} &=& \bar{D} + \frac{1}{2}\sum_{\alpha=1}^{n-1}\nu\cdot e_{\alpha}\cdot \nu\cdot\nabla_{\alpha}\nu\cdot \notag\\ &=& \bar{D} + \frac{1}{2}\sum_{\alpha=1}^{n-1}e_{\alpha}\cdot \nabla_{\alpha}\nu\cdot \notag\\ &=& \bar{D} + \frac{1}{2}\sum_{\alpha,\beta=1}^{n-1}\langle\nabla_{\alpha}\nu,e_{\beta}\rangle e_{\alpha}\cdot e_{\beta}\cdot  \notag\\ &=& \bar{D}-\frac{n-1}{2}H \label{DD},
\end{eqnarray} and $$\nu\cdot\nabla\varphi\cdot =\bar{\nabla}\varphi\ \bar{\cdot} -\nu(\varphi),$$where $\varphi\in C^1(M)$ and $\bar{\nabla}\varphi$ is the gradient of $\varphi|_{\partial{M}}$, we have for every $s\in \Gamma(M,\mathbb{S})$ 
 \begin{eqnarray}\sum_{\alpha=1}^{n-1}\left\langle\nu\cdot e_{\alpha}\cdot\nabla_{\alpha}s,s\right\rangle-(\tau-\delta)\left\langle\nu\cdot\nabla\varphi\cdot s,s\right\rangle = \left\langle\bar{D}^*_{(\tau-\delta)\varphi}s,s\right\rangle + \left((\tau-\delta)\nu(\varphi)-\frac{n-1}{2}H\right)|s|^2.\label{D}
 \end{eqnarray}

\begin{lemma}\label{APS1}Suppose $s\in\Gamma(M,\mathbb{S})$ has the property that $e^{(\delta-\tau)\varphi}s|_{\partial{M}}$ satisfies the $b$-APS boundary condition for some $b\in\mathbb{R}$, then we have
\begin{eqnarray}
\left\Vert D_{\tau\varphi }^{\ast }s\right\Vert_{2(\tau-\delta)\varphi}^{2}&\geq& \frac{(\xi-1)\delta}{\xi}\int_M\left(\Delta\varphi+\frac{(\xi+1)\delta}{\xi}\vert\nabla\varphi\vert^2\right)\vert s\vert^2e^{2(\delta-\tau)\varphi} \notag\\ && +(1-\xi)\int_M\left(\frac{1}{\eta_1^2}\vert P_{\eta,(\tau+(\frac{1}{\xi}-1)\delta)\varphi}s\vert^2+\langle s,\mathfrak{R}s\rangle\right)e^{2(\delta-\tau)\varphi}\notag \\&& +(1-\xi)\int_{\partial{M}}\left(\frac{n-1}{2}H-b+\frac{\delta}{\xi}\nu(\varphi)\right)\vert s\vert^2 e^{2(\delta-\tau)\varphi}\label{WW2}.
\end{eqnarray}
Moreover, the equality holds if and only if  $s|_{\partial{M}} = 0$ or $e^{(\delta-\tau)\varphi}s|_{\partial{M}}$ is a $b$-eigensection of $\bar{D}$. 
If we choose $\eta_1=n\eta_2\neq 0$, then \begin{eqnarray}
\left\Vert D_{\tau\varphi }^{\ast }s\right\Vert_{2(\tau-\delta)\varphi}^{2}&\geq& n\delta\int_M\big(\Delta\varphi+(2-n)\delta\nabla\varphi\vert^2\big)e^{2(\delta-\tau)\varphi}
\notag \\&& +\frac{n}{n-1}\int_M\big(\vert P_{(\tau-n\delta)\varphi}s\vert^2+\langle s,\mathfrak{R}s\rangle\big)e^{2(\delta-\tau)\varphi}
\notag \\ && +n\int_{\partial{M}}\left(\frac{H}{2}-\frac{b}{n-1}-\delta\nu(\varphi)\right)\vert s\vert^2e^{2(\delta-\tau)\varphi}.\label{WW2'}
\end{eqnarray}
\end{lemma}
\begin{proof}
Substituting (\ref{D}) into (\ref{weighted-L2-est1}), the boundary term is given by 
\begin{eqnarray}
\int_{\partial{M}}\left\langle\bar{D}_{(\tau-\delta)\varphi}^*s,s\right\rangle e^{-2(\tau-\delta)\varphi} =\int_{\partial{M}}\left\langle\bar{D}(e^{(\delta-\tau)\varphi}s),e^{(\delta-\tau)\varphi}s\right\rangle\leq b\int_{\partial{M}}\vert s\vert^2e^{-2(\tau-\delta)\varphi} \notag
\end{eqnarray}where we have used, in the last inequality, the assumption that $e^{(\delta-\tau)\varphi}s|_{\partial{M}}$ satisfies the $b$-APS boundary condition. The proof is thus complete.
\end{proof}

\begin{definition}\label{MAd}
The modified $b$-APS boundary condition for $s\in L^2(\partial{M},\bar{\mathbb{S}})$ means that $s+\nu\cdot s$ or  $s-\nu\cdot s$ satisfies the $b$-APS boundary condition where $b\in\mathbb{R}$.
\end{definition}

Again, it follows from (\ref{I1'}) that for $b\leq 0$ every eigenvalue $\lambda(D)$ under the (modified) $b$-APS boundary condition is real.

Assume $t\in L^2(\partial{M},\bar{\mathbb{S}})$ satisfies the modified $b$-APS boundary condition. 

When $b\geq 0$, we split $t$ according to the spectral decomposition of $\bar{D}$ $$t=t_{<-b} + t_{[-b,b]} + t_{>b}\ {\rm on} \ \partial{M}$$ where $t_{<b}, t_{(-b,b)},  t_{>b} \in L^2(\partial{M},\bar{\mathbb{S}})$ lies in the space spanned by the eigensections of $\bar{D}$ with eigenvalues in $(-\infty,\-b), [-b,b]$ and $(b,+\infty)$ respectively.

By  (\ref{DD}), we have \begin{equation}
\bar{D}\cdot\nu = -\nu\cdot\bar{D}\label{AC}
\end{equation} from which it follows that the modified $b$-APS boundary condition for $t$ gives $$t_{>b}\pm\nu\cdot t_{<-b} \equiv 0\ {\rm on} \ \partial{M}.$$ Now we can estimate the boundary integral
\begin{eqnarray*}
\left(\bar{D}t,t\right)_{\partial{M}} &=& \left(\bar{D}t_{<-b} ,t_{<-b} \right)_{\partial{M}} + \left(\bar{D}t_{[-b,b]} ,t_{[-b,b]}\right)_{\partial{M}}  + \left(\bar{D}t_{>b} ,t_{>b}\right)_{\partial{M}}\\
&=&  \left(\bar{D}t_{<-b} ,t_{<-b} \right)_{\partial{M}} + \left(\bar{D}t_{[-b,b]} ,t_{[-b,b]}\right)_{\partial{M}}\\ && + \left(\bar{D}(\mp\nu\cdot t_{<-b}) ,\mp\nu\cdot t_{<-b}\right)_{\partial{M}}\\ 
&=& \left(\bar{D}t_{[-b,b]} ,t_{[-b,b]}\right)_{\partial{M}} \leq b\Vert t\Vert^2_{\partial{M}}.
\end{eqnarray*}

When $b<0$, we split $t$ as $$t=t_{\leq b} + t_{(b,-b)} + t_{\geq -b}\ {\rm on} \ \partial{M}$$which implies $$ t_{(b,-b)}\equiv 0, t_{\geq -b}\pm\nu\cdot t_{\leq b}\equiv 0$$ and therefore
\begin{eqnarray*}
\left(\bar{D}t,t\right)_{\partial{M}} &=& \left(\bar{D}t_{\leq b} ,t_{\leq b} \right)_{\partial{M}} + \left(\bar{D}t_{\geq -b} ,t_{\geq -b}\right)_{\partial{M}} \\
&=& \left(\bar{D}t_{\leq b} ,t_{\leq b} \right)_{\partial{M}} + \left(\bar{D}(\mp\nu\cdot t_{\leq b}) ,\mp\nu\cdot t_{\leq b}\right)_{\partial{M}} \\ &=& 0.
\end{eqnarray*}

To summarize, we have 
\begin{equation}
\int_{\partial{M}}\langle \bar{D}t,t\rangle  = \begin{cases}
     0  & \text{if } b\leq 0,\\  \text{} \\
     \left(\bar{D}t_{[-b,b]} ,t_{[-b,b]}\right)_{\partial{M}} \leq b\int_{\partial{M}}\vert t\vert^2 & \text{if } b>0.
    \end{cases} \label{maps}
\end{equation} 
for any $t\in L^2(\partial{M},\bar{\mathbb{S}})$ satisfying the modified $b$-APS boundary condition. With (\ref{maps}) at hand, we can prove the following estimate by the same method for Lemma \ref{APS1}.

\begin{lemma}\label{mAPS}Suppose $s\in\Gamma(M,\mathbb{S})$ has the property that $e^{(\delta-\tau)\varphi}s|_{\partial{M}}$ satisfies the modified $b$-APS boundary condition for some $b\in\mathbb{R}$. If $b\leq 0$,  
\begin{eqnarray}
\left\Vert D_{\tau\varphi }^{\ast }s\right\Vert_{2(\tau-\delta)\varphi}^{2}&=& \frac{(\xi-1)\delta}{\xi}\int_M\left(\Delta\varphi+\frac{(\xi+1)\delta}{\xi}\vert\nabla\varphi\vert^2\right)\vert s\vert^2e^{2(\delta-\tau)\varphi} \notag\\ && +(1-\xi)\int_M\left(\frac{1}{\eta_1^2}\vert P_{\eta,(\tau+(\frac{1}{\xi}-1)\delta)\varphi}s\vert^2+\langle s,\mathfrak{R}s\rangle\right)e^{2(\delta-\tau)\varphi}\notag \\&& +(1-\xi)\int_{\partial{M}}\left(\frac{n-1}{2}H+\frac{\delta}{\xi}\nu(\varphi)\right)\vert s\vert^2 e^{2(\delta-\tau)\varphi}\label{WWW3}.
\end{eqnarray}
If $b>0$, \begin{eqnarray}
\left\Vert D_{\tau\varphi }^{\ast }s\right\Vert_{2(\tau-\delta)\varphi}^{2}&\geq& \frac{(\xi-1)\delta}{\xi}\int_M\left(\Delta\varphi+\frac{(\xi+1)\delta}{\xi}\vert\nabla\varphi\vert^2\right)\vert s\vert^2e^{2(\delta-\tau)\varphi} \notag\\ && +(1-\xi)\int_M\left(\frac{1}{\eta_1^2}\vert P_{\eta,(\tau+(\frac{1}{\xi}-1)\delta)\varphi}s\vert^2+\langle s,\mathfrak{R}s\rangle\right)e^{2(\delta-\tau)\varphi}\notag \\&& +(1-\xi)\int_{\partial{M}}\left(\frac{n-1}{2}H-b+\frac{\delta}{\xi}\nu(\varphi)\right)\vert s\vert^2 e^{2(\delta-\tau)\varphi}\label{WW3}.
\end{eqnarray}
Moreover, the equality holds if and only if  $s|_{\partial{M}} = 0$ or $e^{(\delta-\tau)\varphi}s|_{\partial{M}}$ is a $b$-eigensection of $\bar{D}$. 
\end{lemma}

\section{Eigenvalue estimates.}

When the underlying manifold is of dimension $2,$ we can estimate eigenvalues of the Dirac operator from below in terms of curvature integrals.

\begin{theorem}\label{T1}
Let $\mathbb{S}$ be a Dirac bundle over a compact Riemannian manifold $\left( M,g\right)$ of dimension $2$ and $D$ be the Dirac
operator. We have the following estimate $$\vert\lambda(D)\vert^2\geq \frac{1}{{\rm Vol}(M)}\left(2\int_M \kappa + \int_{\partial{M}} H.\right)$$where $\lambda(D)$ is an arbitrary eigenvalue under the MIT bag boundary condition or the local boundary condition when the boundary is non-empty($\partial{M}=\emptyset$ is also allowed). The equality holds if and only if $\partial{M}$ is minimal and there is a nontrivial section $s\in\Gamma(M,\mathbb{S})$ satisfying the corresponding boundary condition such that $\mathfrak{R}s=\kappa s$ and $\nabla_Xs +\frac{\lambda}{2}X\cdot s=0$ for all tangent vector $X$ and some constant $\lambda\in\mathbb{R}$, moreover we also have $\kappa = \frac{\lambda^2}{2}$ in this case.

\end{theorem}
\begin{proof}
We will use the solution $\varphi$ of the following Neumann boundary problem as our weight function
\begin{eqnarray}
\frac{1}{2}\Delta \varphi + \kappa &=& \frac{1}{{\rm Vol}(M)}\left(\int_M\kappa+\frac{1}{2}\int_{\partial{M}}H\right) \ {\rm on} \ M, \notag \\
\nu(\varphi)&=& H \ {\rm on} \ \partial{M}. \notag
\end{eqnarray}
Choose a nontrivial section  $s\in\Gamma(M,\mathbb{S})$ such that its restriction to boundary satisfies the MIT bag boundary condition or the local boundary condition and $D(e^{-\varphi}s)=\lambda e^{-\varphi}s$. Then (\ref{W1'}) with $\tau= 1,\delta=\frac{1}{2}$ implies that
$$\vert\lambda\vert^2 \Vert s\Vert^2_\varphi \geq \frac{1}{{\rm Vol}(M)}\left(2\int_M\kappa+\int_{\partial{M}}H\right) \Vert s\Vert^2_\varphi$$and that the quality holds only if $
\mathfrak{R}s=\kappa s$ and $Ps=0$ on $M$. From $D(e^{-\varphi}s)=\lambda e^{-\varphi}s$ and $Ps=0$, we have $\nabla_Xs=-\frac{\lambda}{2}X\cdot s-\frac{1}{2}X\cdot\nabla\varphi\cdot s$ for all tangent vector $X$. Combing (\ref{RI}), $\frac{\Delta\varphi}{2}=-\kappa+\frac{\vert\lambda\vert^2}{2}$ and $\mathfrak{R}s=\kappa s$, we obtain $$(\lambda-\bar{\lambda})s+\nabla\varphi\cdot s=0.$$ Multiplying both sides by $\nabla\varphi$, the above identity gives $$(\lambda-\bar{\lambda})^2+\vert\nabla\varphi\vert^2=0$$ and therefore $$\lambda\in\mathbb{R}, \  \nabla\varphi=0.$$ By $\nabla\varphi=0$, we know that  $H=0$ on $\partial{M}$ and $\nabla_Xs +\frac{\lambda}{2}X\cdot s=0$ for all tangent vector $X$. Conversely, if $H=0$ and 
there exists a nontrivial section $s\in\Gamma(M,\mathbb{S})$ satisfying $\mathfrak{R}s=\kappa s$ and $\nabla_Xs +\frac{\lambda}{2}X\cdot s=0$ for all tangent vector $X$ and some constant $\lambda\in\mathbb{R}$. From $\nabla_Xs +\frac{\lambda}{2}X\cdot s=0$, we have $Ds=\lambda s$ and $\mathfrak{R}s=\frac{\lambda^2}{2}s$ where we have used (\ref{RI}) to get the latter identity which implies $\kappa = \frac{\lambda^2}{2}$. 
\end{proof}

\begin{remark}
For the classical Dirac operator on a spin manifold, $4\kappa =$ the scalar curvature $R$, and therefore $2\int_M \kappa + \int_{\partial{M}} H = 2\pi\chi(M).$ This estimate was obtained   by Chen, Wang and Zhang under the local boundary condition(\cite{CWZ}).
\end{remark}
\bigskip
When $\dim M \geq 3$, we have a gradient term in (\ref{W1}) which makes it infeasible to find weight functions in the same way as two dimensional case.
\begin{theorem}\label{Thm2}
Let $M$ be a compact Riemannian manifold with smooth boundary of dimension $n\geq 3$, and $D$ be the 
Dirac operator of a Dirac bundle $\mathbb{S}$ over $M$,  then every eigenvalue $\lambda(D)$ of $D$ under the MIT bag boundary condition or local boundary condition satisfies
\begin{eqnarray*}
\vert\lambda (D)\vert^2\geq\inf_{{\tiny\begin{array}{c} 
u\in C^\infty(M) \\ 
 \int_M u^2=1\\ 
\end{array}}}\left(\int_M\left(\frac{n}{n-2}\vert\nabla u\vert^2+\frac{n}{n-1}\kappa u^2\right)+\frac{n}{2}\int_{\partial{M}}Hu^2\right)
\end{eqnarray*}where $H$ is the mean curvature w.r.t. the outward unit mean curvature of the boundary. The equality holds if and only if $\partial{M}$ is minimal and there is a nontrivial section $s\in\Gamma(M,\mathbb{S})$ satisfying the corresponding boundary condition  such that $\mathfrak{R}s=\kappa s$ and $\nabla_Xs +\frac{\lambda}{n}X\cdot s=0$ for all tangent vector $X$ and some constant $\lambda\in\mathbb{R}$, moreover we also have $\kappa=\frac{n-1}{n}\lambda^2$ in this case. 
\end{theorem}

\begin{proof}
Let $s\in\Gamma(M,\mathbb{S})$ be a nontrivial section and $\varphi\in C^2(M)$ be a real function to be determined  such that $D(e^{-\tau\varphi}s)=\lambda e^{-\tau\varphi}s$ and $s|_{\partial{M}}$ satisfies the MIT bag boundary condition or the local boundary condition. It follows from (\ref{W1'}) that 
\begin{eqnarray}
\vert\lambda\vert^2\Vert s\Vert^2_{2(\tau-\delta)\varphi}&\geq& \int_{M}\left(n\delta\Delta \varphi + n(2-n)\delta^2\left\vert \nabla \varphi \right\vert ^{2}+\frac{n}{n-1}\kappa\right) \left\vert s\right\vert ^{2}e^{2(\delta-\tau)\varphi}\notag \\ && +\frac{n}{n-1}\int_M\vert P_{(\tau-n\delta)\varphi}s\vert^2 e^{2(\delta-\tau)\varphi}+n\int_{\partial{M}}\left(\frac{H}{2}-\delta\nu(\varphi)\right)\vert s\vert^2 e^{2(\delta-\tau)\varphi}\notag \\ &\geq& \int_{M}\vert s\vert^2 e^{(n\delta-2\tau)\varphi}\left(-\frac{n}{n-2}\Delta + \frac{n}{n-1}\kappa\right)e^{(2-n)\delta\varphi} \notag \\ && +n\int_{\partial{M}}\vert s\vert^2e^{(n\delta-2\tau)\varphi}\left(\frac{\nu}{n-2} + \frac{H}{2}\right)e^{(2-n)\delta\varphi}.\label{311}
\end{eqnarray}
To find an appropriate weight function, we consider the problem of minimizing $$I(u):=\int_M\left(\frac{n}{n-2}\vert\nabla u\vert^2+\frac{n}{n-1}\kappa u^2\right)+\frac{n}{2}\int_{\partial{M}}Hu^2$$ for $u\in W^{1,2}(M)$ subject to the constraint condition $\int_M u^2 =1$. The existence of a minimizer for $I$ is well-known, the case where $\kappa=\frac{1}{4}$scalar curvature played an important role in the study of the Yamabe problem on manifolds with boundary(\cite{Es}). For the sake of completeness, we sketch a proof of the existence of a minimizer for $I$. Let $\{u_\ell\}$ be a minimizing sequence for $I$, then 
\begin{itemize}
\item The Sobolev inequality $\int_{\partial{M}}u^2\leq\epsilon\int_M\vert\nabla u\vert^2 + C_\epsilon\int_Mu^2$( for $u\in W^{1,2}(M)$ and $\epsilon>0$) implies that $\{u_\ell\}$ is bounded in $W^{1,2}(M)$.
\item By the Rellich Theorem and the reflexivity of $W^{1,2}(M)$, passing to a subsequence, we may assume that $\{u_\ell\}$ is convergent in $L^2(M)$ and weakly convergent in $W^{1,2}(M)$. Denote by $u\in W^{1,2}(M)$ the above limit, then $\{u_\ell\vert_{\partial{M}}\}$ converges to $u\vert_{\partial{M}}$ in  $L^2(\partial{M})$. Hence, $u$ is a minimizer for $I$.
\end{itemize}
The Lagrange Multiplier Theorem implies that the minimizer $u$ satisfies 
\begin{eqnarray*}
-\frac{n}{n-2}\Delta u + \frac{n}{n-1}\kappa u&=& (\inf I) u \ {\rm on} \ M,  \\
\nu(u)&=&-\frac{n-2}{2}Hu\ {\rm on} \ \partial{M}.
\end{eqnarray*}
Moreover, by the definition of $I$, $\vert u\vert$ also minimize $I$ for any minimizer $u$ and thus is a solution of the above oblique boundary problem. The standard regularity theory shows $\vert u\vert\in C^2(M)$. Now we know by the Hopf Lemma that $\vert u\vert\neq 0$ on $M$, so the above oblique boundary problem has a positive solution $u$ which will be used to define the weight function \begin{equation}\varphi=\frac{\log u}{(2-n)\delta}\label{phi}.
\end{equation} The inequality (\ref{311}) gives $$\vert\lambda\vert^2\Vert s\Vert^2_{2(\tau-\delta)\varphi}\geq \inf(I) \Vert s\Vert^2_{2(\tau-\delta)\varphi},$$ which implies the desired estimate.

Now we assume the equality in the estimate holds. Let $\tau=n\delta=1$ in (\ref{311}), then $\mathfrak{R}s=\kappa s$, 
$\nabla_X s +\frac{1}{n}X \cdot Ds=0$ for all tangent vector $X$ and therefore $\nabla_X s =-\frac{\lambda}{n}X \cdot s-\frac{1}{n}X\cdot\nabla\varphi\cdot s.$ From (\ref{phi}), it follows that  $-\frac{n-1}{n}\Delta\varphi=\frac{(n-1)(2-n)}{n^2}\vert\nabla\varphi\vert^2+\kappa - \frac{n-1}{n}\vert\lambda\vert^2$. Now we can apply (\ref{RI}) to deduce $$(\lambda-\bar{\lambda})s+\frac{2}{n}\nabla\varphi\cdot s=0.$$
By the same argument in the proof of  Theorem \ref{T1}, we have $H=0$ and $\nabla\varphi =0$. Conversely, assume that $H=0$ and 
there exists a nontrivial section $s\in\Gamma(M,\mathbb{S})$ satisfying $\mathfrak{R}s=\kappa s$ and $\nabla_Xs +\frac{\lambda}{n}X\cdot s=0$ for all tangent vector $X$ and some constant $\lambda\in\mathbb{R}$. From $\nabla_Xs +\frac{\lambda}{n}X\cdot s=0$, we have $Ds=\lambda s$ by definition and $\mathfrak{R}s=\frac{(n-1)\lambda^2}{n}s$ by (\ref{RI}). The latter identity, together with the assumption $\mathfrak{R}s=\kappa s$, implies $\kappa = \frac{(n-1)\lambda^2}{n}$. Since $\kappa$ is a constant and $H=0$ in this case, we have \begin{eqnarray*}
&&\inf_{{\tiny\begin{array}{c} 
u\in C^\infty(M) \\ 
 \int_M u^2=1\\ 
\end{array}}}\left(\int_M\left(\frac{n}{n-2}\vert\nabla u\vert^2+\frac{n}{n-1}\kappa u^2\right)+\frac{n}{2}\int_{\partial{M}}Hu^2\right)\ \\ &&=\inf_{{\tiny\begin{array}{c} 
u\in C^\infty(M) \\ 
 \int_M u^2=1\\ 
\end{array}}}\int_M\frac{n}{n-2}\vert\nabla u\vert^2 + \lambda^2  =\lambda^2.
\end{eqnarray*}
The proof is complete.\end{proof}

\begin{theorem}\label{Thmabs}
Let $\mathbb{S}$ be a Dirac bundle over a compact Riemannian manifold $\left( M,g\right)$ of dimension $n\geq 2$ and $D$ be the Dirac
operator. Then every eigenvalue $\lambda(D)$  under the $b$-APS boundary condition satisfies

\begin{itemize}
\item If $b\leq 0$, $$\lambda(D)^2\geq \inf_{{\tiny\begin{array}{c} 
u\in C^\infty(M) \\ 
 \int_M u^2=1\\ 
\end{array}}}\frac{n}{n-1}\left(\int_M\Big(\frac{n}{n-1}\vert\nabla u\vert^2+\kappa u^2\Big)+\int_{\partial{M}}\Big(\frac{n-1}{2}H-b\Big)u^2\right).$$ 
\item If $b>0$, $$\\ \vert\lambda(D)\vert^2\geq \inf_{{\tiny\begin{array}{c} 
u\in C^\infty(M) \\ 
 \int_M u^2=1\\ 
\end{array}}}\left(\int_M\Big(\vert\nabla u\vert^2+\kappa u^2\Big)+\int_{\partial{M}}\Big(\frac{n-1}{2}H-b\Big)u^2\right).$$
The equality holds if and only if $H=\frac{2b}{n-1},\kappa=0$ and there is some nontrivial parallel section $s\in\Gamma(M,\mathbb{S})$ satisfying the $b$-APS boundary condition, moreover $s|_{\partial{M}}$ must be a $b$-eigensection of $\bar{D}$ in this case.

\end{itemize}
\end{theorem}
\begin{proof}
Suppose that the restriction of $e^{(\delta-\tau)\varphi}s$ to $\partial{M}$ satisfies the $b$-APS boundary condition and that 
$D(e^{(\delta-\tau)\varphi}s)=\lambda e^{(\delta-\tau)\varphi}s,$ where $s\in\Gamma(M,\mathbb{S})$ is a nontrivial section and $\varphi\in C^2(M)$ is a real function to be determined. 

Assume $b\leq 0$, as $\lambda$ is real in this case, we have  
\begin{eqnarray}\vert D_{\tau\varphi}^*s\vert^2 = \vert D_{(\tau-\delta)\varphi}^*s-\delta\nabla\varphi\cdot s\vert^2 = \vert \lambda s-\delta\nabla\varphi\cdot s\vert^2 = \vert\lambda\vert^2\vert s\vert^2+\delta^2\vert\nabla\varphi\vert^2\vert s\vert^2. \label{312} 
\end{eqnarray} 
Combining  (\ref{WW2'}) and (\ref{312}) gives
\begin{eqnarray}
\lambda^2\Vert s\Vert_{2(\tau-\delta)\varphi}^{2}&\geq& \int_{M}\left( n\delta\Delta \varphi - (n-1)^2\delta^2 \left\vert \nabla \varphi \right\vert ^{2}+\frac{n}{n-1}\kappa
\right) \left\vert s\right\vert ^{2}e^{2(\delta-\tau)\varphi}\notag \\ && +\frac{n}{n-1}\int_M\vert P_{(\tau-n\delta)\varphi}s\vert^2e^{2(\delta-\tau)\varphi}\notag \\ &&+n\int_{\partial{M}}\left(\frac{H}{2}-\frac{b}{n-1}-\delta\nu(\varphi)\right)\vert s\vert^2e^{2(\delta-\tau)\varphi} \notag \\ &\geq& \int_{M}\vert s\vert^2e^{\big(\frac{n^2+1}{n}\delta-2\tau\big)\varphi}\left(-\frac{n^2}{(n-1)^2}\Delta + \frac{n}{n-1}\kappa\right)e^{-\frac{(n-1)^2\delta}{n}\varphi} \label{313}\\ && +n\int_{\partial{M}}\vert s\vert^2e^{\big(\frac{n^2+1}{n}\delta-2\tau\big)\varphi}\left(\frac{n}{(n-1)^2}\nu +\frac{H}{2}-\frac{b}{n-1}\right)e^{-\frac{(n-1)^2\delta}{n}\varphi}.\notag
\end{eqnarray}
By the same argument as in the proof of Theorem \ref{Thm2},  we choose a positive minimizer $u\in C^2(M)$ for the variation integral $$I(u):=\int_M\left(\frac{n^2}{(n-1)^2}\vert\nabla u\vert^2+\frac{n}{n-1}\kappa u^2\right)+n\int_{\partial{M}}\left(\frac{H}{2}-\frac{b}{n-1}\right)u^2$$ subject to the constraint condition $\int_M u^2 =1$. Again, the Lagrange Multiplier Theorem gives
\begin{eqnarray}
-\frac{n^2}{(n-1)^2}\Delta u + \frac{n}{n-1}\kappa u&=& \inf(I) u \ {\rm on} \ M, \notag \\
\nu(u)&=& \frac{(n-1)^2}{n}\left(\frac{b}{n-1}-\frac{H}{2}\right)u \ {\rm on} \ \partial{M}. \label{314}
\end{eqnarray}
Choosing the weight function $\varphi = -\frac{n}{(n-1)^2\delta}\log u$, then the desired estimate follows from (\ref{313}).

Assume that the equality holds and $\lambda\neq 0$. Let $\tau=\delta=\frac{1}{1-n}$, then $Ds=\lambda s$. By (\ref{313}), we have $\bar{D}s|_{\partial{M}}=bs|_{\partial{M}}, \mathfrak{R}s=\kappa s$ and $P_{\varphi}s=0$. The last one implies\begin{equation}
\nabla_X s = -\frac{\lambda}{n}X\cdot s + \frac{1}{n}X\cdot\nabla\varphi\cdot s + X(\varphi) s\label{APSE}
\end{equation}
for all tangent vector $X$. Substituting $-\frac{n}{n-1}\Delta\varphi = \vert\nabla\varphi\vert^2-\frac{n\kappa}{n-1}+\lambda^2$ and (\ref{APSE}) into (\ref{RI}), we obtain $\big((n-1)\lambda^2 -(n-1)\vert\nabla\varphi\vert^2-n\kappa\big)s=2\lambda \nabla\varphi\cdot s$ which in turn implies $\nabla\varphi =0$, and therefore $\nabla_Xs=-\frac{\lambda}{n}X\cdot s$(by (\ref{APSE})). Now it follows from (\ref{DD}) that $(b-\frac{n-1}{2}H)s|_{\partial{M}} = \frac{n-1}{n}\lambda\nu\cdot s|_{\partial{M}}$ which, multiplying both sides by $\nu$, gives $$\lambda =0, b=\frac{n-1}{2}H.$$By $\mathfrak{R}s=\kappa s$ and $\nabla_Xs=0$, we also have $\kappa =0$. The converse direction is clear, so we have finished the proof of the case $b\leq 0$.

When $b>0$, to avoid the possible mixed term in (\ref{312}), we use a variant of Lemma \ref{mAPS} to allow $\delta=0$ which is given by  Corollary \ref{weighted-L2-coro}. The rest is parallel to the proof  in the case $b\leq 0$.   
\end{proof}

\bigskip

The same argument based on Lemma \ref{mAPS} gives the estimate under the modified $b$-APS boundary condition. 
\begin{theorem}\label{Thmmabs}
Let $\mathbb{S}$ be a Dirac bundle over a compact Riemannian manifold $\left( M,g\right)$ of dimension $n\geq 2$ and $D$ be the Dirac
operator. Then every eigenvalue $\lambda(D)$  under the modified $b$-APS boundary condition satisfies

\begin{itemize}
\item If $b\leq 0$, $$\lambda(D)^2\geq \inf_{{\tiny\begin{array}{c} 
u\in C^\infty(M) \\ 
 \int_M u^2=1\\ 
\end{array}}}\frac{n}{n-1}\left(\int_M\Big(\frac{n}{n-1}\vert\nabla u\vert^2+\kappa u^2\Big)+\frac{n-1}{2}\int_{\partial{M}}Hu^2\right).$$ 
The equality holds if and only if $\partial{M}$ is minimal and there is a nontrivial section $s\in\Gamma(M,\mathbb{S})$ satisfying the modified $b$-APS boundary condition such that $\mathfrak{R}s=\kappa s$ and $\nabla_Xs +\frac{\lambda}{n}X\cdot s=0$ for all tangent vector $X$ and some constant $\lambda\in\mathbb{R}$, moreover we also have $\kappa=\frac{n-1}{n}\lambda^2$ in this case. 
\item If $b>0$, $$\\ \vert\lambda(D)\vert^2\geq \inf_{{\tiny\begin{array}{c} 
u\in C^\infty(M) \\ 
 \int_M u^2=1\\ 
\end{array}}}\left(\int_M\Big(\vert\nabla u\vert^2+\kappa u^2\Big)+\int_{\partial{M}}\Big(\frac{n-1}{2}H-b\Big)u^2\right).$$
The equality holds if and only if $H=\frac{2b}{n-1},\kappa=0$ and there is some nontrivial parallel section $s\in\Gamma(M,\mathbb{S})$ satisfying the $b$-APS boundary condition, moreover $s|_{\partial{M}}$ must be a $b$-eigensection of $\bar{D}$ in this case.

\end{itemize}
\end{theorem}

Now we deduce a lower bound in terms of the volume of the underlying manifold assuming the curvature $\mathfrak{R}$ of $\mathbb{S}$ and the mean curvature $H$ of $\partial{M}$ are bounded from below by constants depending on $(M,g)$.
Li and Zhu(\cite{LZ}) proved the following sharp Sobolev inequality.
Let $(M,g)$ be a Riemannian manifold, $n:=\dim M\geq 3$, then for all for all $u\in C^\infty(M)$,

\begin{eqnarray}
\left(\int_M \vert u\vert^{\frac{2n}{n-2}}dV_g\right)^{\frac{n-2}{n}}\leq S\int_M\vert\nabla u\vert^2 dV_g+ A\left(\int_M u^2dV_g +\int_{\partial{M}} u^2 dA_g\right)\label{LZ}
\end{eqnarray} where 
$S_n=\frac{1}{\pi n(n-2)}\left(\frac{2\Gamma(n)}{\Gamma(\frac{n}{2})}\right)^{\frac{2}{n}}$ and $A=A(M,g)$ is a positive constant. 

\begin{corollary}
Let $M$ be a compact Riemannian manifold with smooth boundary of dimension $n\geq 3$, $D$ be the 
Dirac operator of a Dirac bundle $\mathbb{S}$ over $M$
Every eigenvalue $\lambda(D)$ of $D$ under the $b$-APS boundary condition or modified $b$-APS boundary condition satisfies

\begin{itemize}
\item If $b\leq 0$, $\kappa\geq\frac{n\gamma}{(n-1)S_n}$ and $H\geq \frac{2n^2\gamma}{(n-1)^3S_n}$, then \begin{eqnarray}
\lambda (D)^2&\geq&\frac{n^2}{(n-1)^2S_n{\rm Vol}(M,g)^{\frac{2}{n}}}.\label{vol}\end{eqnarray}
\item If $b>0$, $\kappa\geq\frac{\gamma}{S_n}$ and $H>\frac{2\gamma}{(n-1)S_n}$, then \begin{eqnarray}
\vert\lambda (D)\vert^2&\geq&\frac{1}{S_n{\rm Vol}(M,g)^{\frac{2}{n}}}.\label{vol'}\end{eqnarray}
\end{itemize}
where $$\gamma:=\sup_{{\tiny\begin{array}{c} 
u\in C^\infty(M) \\ 
 \int_M u^2+\int_{\partial{M}}u^2=1\\ 
\end{array}}}\left({\rm Vol}(M,g)^{-\frac{2}{n}}\int_M u^2 - S_n\int_M\vert\nabla u\vert^2\right)$$ and $S_n=\frac{1}{\pi n(n-2)}\left(\frac{2\Gamma(n)}{\Gamma(\frac{n}{2})}\right)^{\frac{2}{n}}$.
\end{corollary}

\begin{proof}
By Li-Zhu inequality (\ref{LZ}), we have $\gamma<\infty$. It follows from the definition of $\gamma$ that \begin{eqnarray}
{\rm Vol}(M,g)^{-\frac{2}{n}}\int_M u^2dV_g\leq S_n\int_M\vert\nabla u\vert^2 dV_g+ \gamma\left(\int_M u^2dV_g +\int_{\partial{M}} u^2 dA_g\right)\label{SI}
\end{eqnarray} holds for all $u\in C^\infty(M)$. The estimate (\ref{vol}) is direct consequence of (\ref{SI}) and Theorems \ref{Thmabs}, \ref{Thmmabs}.

We prove (\ref{vol'}) by a rescaling argument. For a given constant $\sigma\neq 0$, we rewrite (\ref{SI}) in terms of the homethetic metric $\sigma^2 g$ as follows   
\begin{eqnarray}\sigma^{-2}{\rm Vol}(M,g)^{-\frac{2}{n}}\int_M u^2dV_{\sigma^2g}&=&
{\rm Vol}(M,\sigma^2g)^{-\frac{2}{n}}\int_M \vert u\vert^2dV_{\sigma^2g}\notag \\&\leq& S_n\int_M\vert\nabla_{\sigma^2g} u\vert_{\sigma^2g}^2 dV_{\sigma^2g}+ \frac{\gamma}{\sigma^{2}}\int_M u^2dV_{\sigma^2g}\notag \\ && +\frac{\gamma}{\sigma}\int_{\partial{M}} u^2 dA_{\sigma^2g}.\label{SI'}
\end{eqnarray}

On the other hand, it is straightforward to see that $(\mathbb{S},\langle,\rangle,\nabla)$ forms a Dirac bundle over $(M,\sigma^2 g)$ with the Clifford multiplication $\cdot_{\sigma}$ defined by $$X\cdot_{\sigma}s:=\sigma X\cdot s$$ for any tangent vector $X$ of $M$ and section $s$ of $\mathbb{S}$. The associated Dirac operator $D_\sigma$ is thus given by \begin{eqnarray}
D_\sigma = \frac{1}{\sigma}D.\label{D'}\end{eqnarray}Similarly, the curvature in (\ref{R-operator}) is given by \begin{eqnarray}
\mathfrak{R}_\sigma = \frac{1}{\sigma^2}\mathfrak{R},\label{R'}\end{eqnarray}and the mean curvature of $\partial{M}$ w.r.t. $\sigma^2g$ is given by \begin{eqnarray}
H_\sigma = \frac{1}{\sigma}H.\label{H'}\end{eqnarray}

Theorems \ref{Thmabs}, \ref{Thmmabs} applied to the Dirac bundle $(\mathbb{S},\langle,\rangle,\nabla)$ over $(M,\sigma^2 g)$ imply
\begin{eqnarray}
\sigma^{-2}\vert\lambda (D)\vert^2&=&\vert\lambda (D_\sigma)\vert^2\notag \\ &\geq&\inf_{{\tiny\begin{array}{c} 
u\in C^\infty(M) \\ 
 \int_M u^2dV_{\sigma^2g}=1\\ 
\end{array}}}\left(\int_M\left(\vert\nabla_{\sigma^2g} u\vert_{\sigma^2g}^2+\kappa_\sigma u^2\right)dV_{\sigma^2g}+\int_{\partial{M}}\Big(\frac{n-1}{2}H_\sigma-b\Big) u^2dA_{\sigma^2g}\right) \notag\\ &=&\inf_{{\tiny\begin{array}{c} u\in C^\infty(M) \\ 
 \int_M u^2dV_{\sigma^2g}=1\\ 
\end{array}}}\left(\int_M\left(\vert\nabla_{\sigma^2g} u\vert_{\sigma^2g}^2+\frac{\kappa}{\sigma^2} u^2\right)dV_{\sigma^2g}+\int_{\partial{M}}\Big(\frac{(n-1)H}{2\sigma}-b\Big) u^2dA_{\sigma^2g}\right).\notag\\ \label{res}
\end{eqnarray}
By the assumption $H>\frac{2\gamma}{(n-1)S_n}$, we know that $\frac{(n-1)H}{2\sigma}-b\geq \frac{\gamma}{\sigma S_n}$ holds for sufficiently small $\sigma>0$.  Fixing such a small $\sigma>0$, the following estimate follows from  (\ref{SI'}) and (\ref{res})
\begin{eqnarray*}
\sigma^{-2}\vert\lambda (D)\vert^2&\geq& S_n^{-1}\sigma^{-2}{\rm Vol}(M,g)^{\frac{2}{n}}
\end{eqnarray*}i.e.,\begin{eqnarray*}
\vert\lambda (D)\vert^2&\geq&\frac{1}{S_n{\rm Vol}(M,g)^{\frac{2}{n}}},
\end{eqnarray*}
which concludes the proof.\end{proof}


{\small Addresses:}

{\small Qingchun Ji}

{\small School of Mathematics, Fudan University}

{\small Shanghai Center for Mathematical Sciences} 

{\small Shanghai 200433, China }

{\small Email: qingchunji@fudan.edu.cn}

\bigskip

{\small Li Lin}

{\small School of Mathematics, Fudan University}

{\small Shanghai Center for Mathematical Sciences} 

{\small Shanghai 200433, China }

{\small Email: 18110840002@fudan.edu.cn}

\end{document}